\documentclass[a4paper]{article}
\usepackage{makecell}
\usepackage{bbm}

\makeatletter
\let\c@figure\c@table
\let\ftype@figure\ftype@table

\usepackage{enumitem}
\usepackage{calc}

\usepackage{subcaption}

\usepackage{lmodern} 

\usepackage{graphicx} 

\usepackage{amsmath, accents}
\usepackage{amssymb,tikz}
\usepackage{pict2e}
\usepackage{amsthm}
\usepackage{amscd}
\usepackage{mathrsfs}
\usepackage{todonotes}
\usepackage{pgfplots}

\usetikzlibrary{calc} 

\usepackage{yfonts}

\usepackage{marvosym}
\usepackage[percent]{overpic}

\newtheorem{theorem}{Theorem} [section]
	
\newtheorem{corollary}[theorem]{Corollary}
\newtheorem{lemma}[theorem]{Lemma}

\theoremstyle{definition}
\newtheorem{definition}[theorem]{Definition}
\newtheorem{example}{Example}[section]
\newtheorem{remark}[theorem]{Remark}

\newcommand{\C}{\mathbb{C}}

\newcommand{\R}{\mathbb{R}}
\newcommand{\N}{\mathbb{N}}

\newcommand{\E}{\mathbb{E}}

\newcommand{\Var}{\operatorname{Var}}
\newcommand{\Cov}{\operatorname{Cov}}

\let\oldbibliography\thebibliography
\renewcommand{\thebibliography}[1]{\oldbibliography{#1}
\setlength{\itemsep}{-0.5pt}}

\def\XXint#1#2#3{{\setbox0=\hbox{$#1{#2#3}{\int}$}
\vcenter{\hbox{$#2#3$}}\kern-.5\wd0}}

\usepackage{tikz}
\usetikzlibrary{arrows}
\usetikzlibrary{decorations.pathmorphing}
\usetikzlibrary{decorations.markings}
\usetikzlibrary{patterns}
\usetikzlibrary{automata}
\usetikzlibrary{positioning}
\usepackage{tikz-cd}
\tikzset{->-/.style={decoration={
				markings,
				mark=at position #1 with {\arrow{latex}}},postaction={decorate}}}
	
	\tikzset{-<-/.style={decoration={
				markings,
				mark=at position #1 with {\arrowreversed{latex}}},postaction={decorate}}}
				
\tikzset{
	master/.style={
		execute at end picture={
			\coordinate (lower right) at (current bounding box.south east);
			\coordinate (upper left) at (current bounding box.north west);
		}
	},
	slave/.style={
		execute at end picture={
			\pgfresetboundingbox
			\path (upper left) rectangle (lower right);
		}
	}
}

\usetikzlibrary{shapes.misc}\tikzset{cross/.style={cross out, draw, 
         minimum size=2*(#1-\pgflinewidth), 
         inner sep=0pt, outer sep=0pt}}

\allowdisplaybreaks	

\numberwithin{equation}{section}

\makeatletter
\newcommand{\oset}[3][0ex]{%
  \mathrel{\mathop{#3}\limits^{
    \vbox to#1{\kern-2\ex@
    \hbox{$\scriptstyle#2$}\vss}}}}
\makeatother

\usepackage{tocloft}
\usepackage[colorlinks=true]{hyperref}
\hypersetup{urlcolor=blue, citecolor=red, linkcolor=blue}

\begin{document}
\title{Two-dimensional Coulomb gases with multiple outposts}
\author{Kohei Noda}

\maketitle

\begin{abstract}
We study two-dimensional Coulomb gases in the presence of $m\in\mathbb{N}_{>0}$ outposts. An outpost is a connected component of the coincidence set that lies outside the droplet.

The case $m=1$ was previously investigated by Ameur, Charlier, and Cronvall. They showed that, as the total number of particles in the Coulomb gas tends to infinity, the number of particles accumulating near the outpost remains of order one and converges in distribution to the Heine distribution.

In this work, we extend this analysis to the case of an arbitrary but fixed number $m$ of outposts. We prove that the joint distribution of the numbers of particles near the outposts converges to a multidimensional Heine distribution. Our results reveal a interesting phenomenon: although the outposts are geometrically disconnected, the particle count near each outpost is strongly correlated with the particle counts near all other outposts, not only the nearest ones (provided the outposts are not separated by a component of the droplet).
\end{abstract}



{
\hypersetup{linkcolor=black}

}
\section{Introduction and statement of results}
\label{section: introduction}
For $n\in\mathbb{N}_{>0}$, consider the two-dimensional determinantal Coulomb gas $\{z_j\}_{j=1}^n\subset\C^{n}$ defined by 
\begin{align}
\label{def of det point process}
d\mathbb{P}_{n}(z_1,\dots,z_n)
&:=\frac{1}{Z_{n}} \prod_{1 \leq j < k \leq n} |z_{k} -z_{j}|^{2} \prod_{j=1}^{n}e^{-n Q(z_{j})}\,\frac{d^2z_j}{\pi}, 
\end{align}
where $d^2z=dxdy$ the area Lebesgue measure on $\C$, and $Z_{n}$ is called the partition function given by the $n$-fold integral 
\begin{align*}
Z_{n}&:=\int_{\C^n}\prod_{1 \leq j < k \leq n} |z_{k} -z_{j}|^{2} \prod_{j=1}^{n}e^{-n Q(z_{j})}\,\frac{d^2z_j}{\pi}.   
\end{align*}
The probability measure \eqref{def of det point process} represents the joint probability distribution of the eigenvalues of a random normal matrix, see e.g., \cite{BFreview}. 
The function $Q:\C\to\R\cup\{+\infty\}$ is called the {\it external potential}. We assume that $Q$ is rotation-invariant, i.e., $Q(z)=q(|z|)$ for some $q:[0,+\infty)\to \R\cup\{+\infty\}$, and confining in the sense that $\liminf_{z\to+\infty}\frac{Q(z)}{2\log |z|}>1$ to ensure that $Z_{n}<+\infty$. 
Assuming that $Q$ is lower semi-continuous and finite on some set of positive capacity, Frostman's theorem \cite{SaTo} implies the existence of a unique {\it equilibrium measure} $\sigma$, which minimizes the weighted logarithmic energy 
\begin{equation}
    I_Q[\mu]:=\int_{\C^2}\log\frac{1}{|z-w|}\,d\mu(z)d\mu(w)+\int_{\C}Q(z)\,d\mu(z),
\end{equation}
over all compactly supported Borel probability measures $\mu$. 
The support of the equilibrium measure is called the {\it droplet} and is denoted $S\equiv S_Q:=\mathrm{supp}(\sigma)$. 

If $Q$ is $C^2$-smooth in a neighbourhood of $S$, then $\sigma$ is absolutely continuous with respect to the Lebesgue measure, and takes the form 
\begin{equation}
    d\sigma(z):=\Delta Q(z)\mathbf{1}_S(z)\frac{d^2z}{\pi},
\end{equation}
where $\Delta$ is the quarter Laplacian defined by $\Delta:=\partial\overline{\partial}=\frac{1}{4}(\partial_x^2+\partial_y^2)$ and $\mathbf{1}_S$ for the characteristic function of $S$. 

Let us define the obstacle function $\check{Q}(z)$ to be the pointwise supremum of $s(z)$, where $s$ runs through the class of subharmonic functions $s:\C\to\R$ which satisfy $s\leq Q$ on $\C$ and $s(z)\leq 2\log |z|+\mathcal{O}(1)$ as $|z|\to+\infty$. 
Clearly, $\check{Q}(z)$ is a sub-harmonic function such that $\check{Q}\leq Q$ and $\check{Q}(z)=2\log |z|+\mathcal{O}(1)$ as $|z|\to+\infty$. 

The {\it coincidence set} $S^{\ast}=S^{\ast}_Q$ for the obstacle problem is defined by 
\begin{equation}
    S^{\ast}=\{z\in\C:Q(z)=\check{Q}(z)\}.
\end{equation}

We assume throughout the rest of this work that $Q$ is $C^{6}$-smooth in a neighborhood of $S^{\ast}$. 
Following \cite{ACC2023b}, we refer to points of $S^{\ast}\backslash S$ as {\it shallow points}, and call a connected component of $S^{\ast}\backslash S$ an {\it outpost} of the droplet.

In our setting, i.e., in the case where the function $Q$ is rotation-invariant, a connected component of $S^{\ast}$ is either a disk $\mathbb{D}_{r}:=\{z:|z|\leq r\}$, an annulus $\mathbb{A}(a,b):=\{z\in\C:a\leq |z|\leq b\}$, the singleton $\{0\}$, or a circle $\{z\in\C:|z|=t\}$.
As a mild restriction, we will assume that $S^{\ast}$ has only finitely many connected components. 
Thus, the droplet $S$ is of the form
\begin{equation}
\label{def of S general}
    S=\bigcup_{\nu=0}^{N}A(a_{\nu},b_{\nu}),
\end{equation}
where $0\leq a_0<b_0<a_1<b_1<\cdots<a_N<b_N$, and $S^{\ast}$ is obtained by possibly adjoining finitely many outposts $\{z\in\C:|z|=t_{p}\}$, where $t_p\geq0$. 
In this work, we consider the following two cases (see Figure~\ref{Fig_Outpost}): 
\begin{description}
    \item[\textbf{Case 1}\label{item:case1}:] $N=0$ and $0\leq a_0<b_0<t_1<\cdots<t_m$, 
    \begin{equation}
    \label{def of droplet case 1}
        S^{\ast}=A(a_0,b_0)\cup \bigcup_{p=1}^{m}\{z\in\C:|z|=t_{p}\}. 
    \end{equation}
    \item[\textbf{Case 2}\label{item:case2}:] $N=1$ and $0\leq a_0<b_0<t_1<\cdots<t_m<a_1<b_1$, 
    \begin{equation}
    \label{def of droplet case 2}
        S^{\ast}=A(a_0,b_0)\cup \bigcup_{p=1}^{m}\{z\in\C:|z|=t_{p}\}\cup A(a_1,b_1). 
    \end{equation}
\end{description}

\begin{figure}[t]
	\begin{subfigure}{0.44\textwidth}
    	\begin{center}	
    		\includegraphics[width=\textwidth]{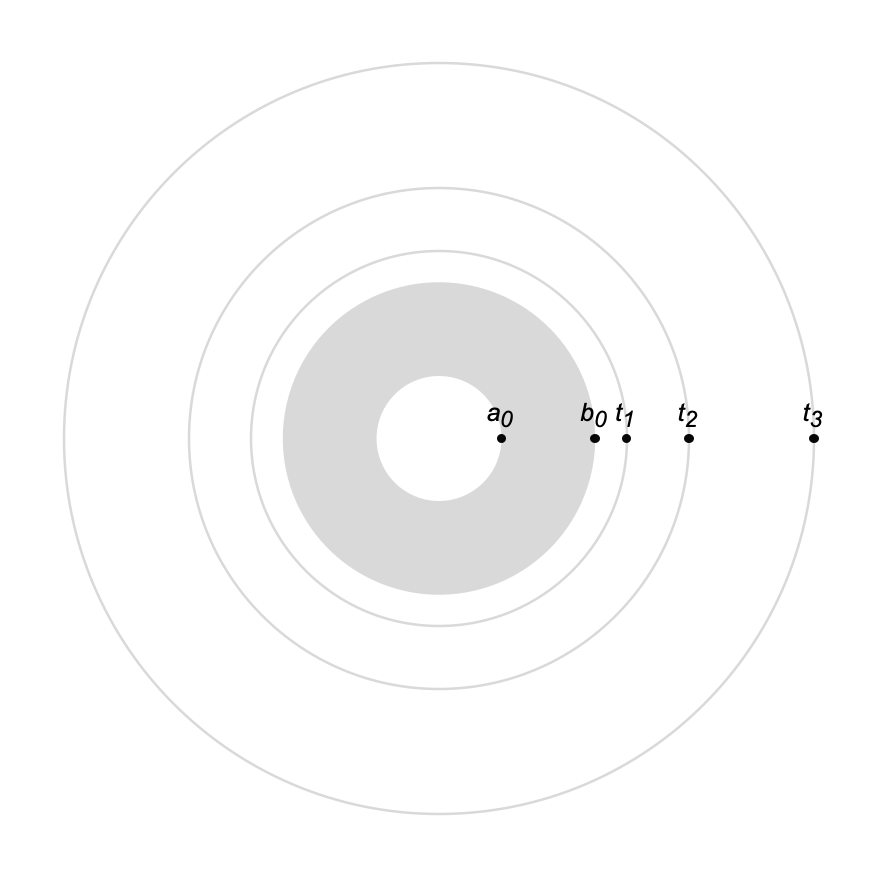}
            \subcaption{Case 1 (with $m=3$)}
    	\end{center}
    \end{subfigure}	 \quad 
	\begin{subfigure}{0.44\textwidth}
	\begin{center}	
		\includegraphics[width=\textwidth]{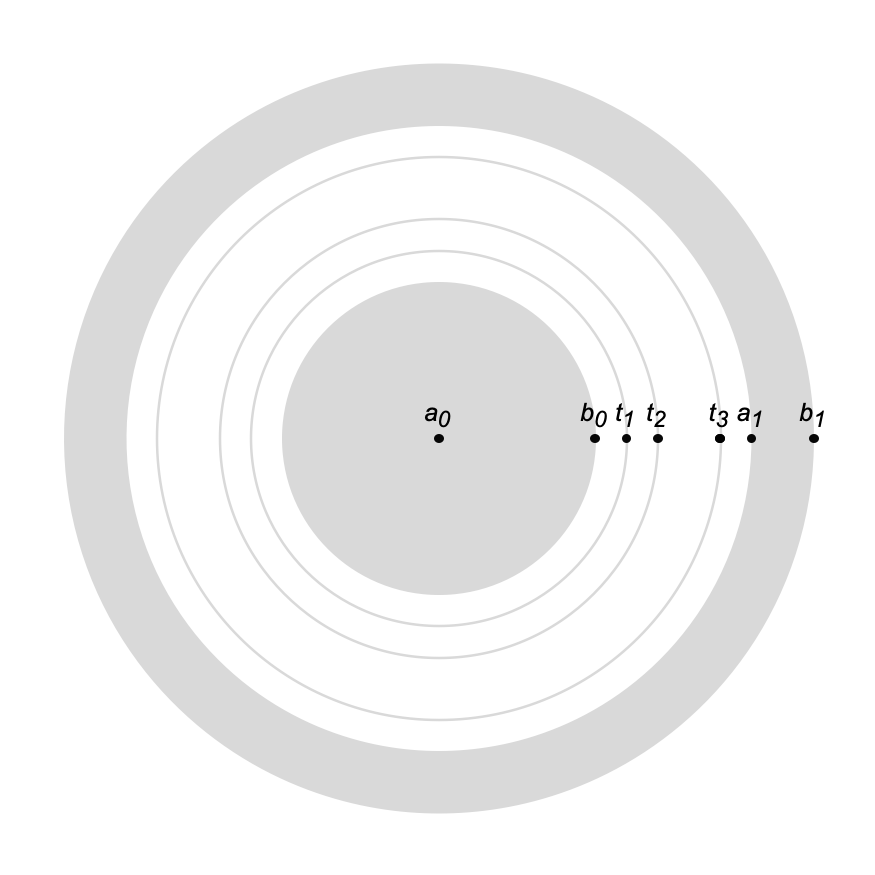}
            \subcaption{Case 2 (with $m=3$)}
	\end{center}
\end{subfigure}	
	\caption{Two cases of $S^{\ast}$
} \label{Fig_Outpost}
\end{figure} 
 
To explain why we focus on the cases $N\in\{0,1\}$, let us first consider the situation $S=S^{\ast}$ as in \eqref{def of S general}, that is, when the droplet has no outposts. 
Let $f$ be a smooth and rotation-invariant test function such that $f\equiv 1$ on an open neighborhood of $S$ and $f\equiv 0$ otherwise. 
Then, by \cite[Corollary~1.3]{ACC2023b}, as $n\to+\infty$,
\begin{equation}
    \label{def of decouple}
\E\Bigl[e^{s\sum_{j=1}^n f(z_j)}\Bigr]
=
\E\Bigl[e^{s\mathcal{N}(e_f,v_f)}\Bigr]
\cdot
\prod_{\nu=0}^{N-1} 
\E\Bigl[e^{s\sum_{j=1}^n f_{\nu}(z_j)}\Bigr]\cdot\,(1+o(1)),
\end{equation}
uniformly for $s$ in compact subsets of $\mathbb{C}$, where $f_{\nu}$ for $\nu=0,1,\dots,N-1$ are smooth, rotation-invariant test functions such that $f_{\nu}\equiv1$ in neighborhoods of the boundaries $|z|=b_{\nu}$ of $A(a_{\nu},b_{\nu})$ and $|z|=a_{\nu+1}$ of $A(a_{\nu+1},b_{\nu+1})$ and $f_{\nu}\equiv0$ otherwise. 
Here, $\mathcal{N}(e_f, v_f)$ denotes a real Gaussian random variable with mean $e_f$ and variance $v_f$; see \cite[(1.14), (1.15)]{ACC2023b} for explicit expressions.
Formula~\eqref{def of decouple} has the interpretation that the fluctuations of smooth linear statistics for $S=S^{\ast}$ as in \eqref{def of S general} are decoupled into a Gaussian part and contributions coming from neighborhoods of the boundaries of each gap, which are asymptotically independent.

Therefore, it suffices to consider two cases: \hyperref[item:case1]{case~1}, where the outposts are outside the outermost connected component of the droplet, and \hyperref[item:case2]{case~2}, where the outposts are located inside a spectral gap between two connected components of the droplet.
(The case where $a_0>0$ and the outposts are inside the innermost connected component is similar to \hyperref[item:case1]{case~1} and is therefore not considered in this work.) 

For concreteness, we now provide an explicit example of a potential $Q$ satisfying \hyperref[item:case1]{case 1}. (Explicit examples of potentials $Q$ satisfying \hyperref[item:case2]{case 2} can be constructed similarly.)
\begin{example}
Consider the quadratic (Ginibre) potential
\[
Q_{\text{g}}(z) = |z|^2.
\]
Its obstacle function is given by
\[
\check{Q}_{\text{g}}(z) = 1 + 2 \log |z|.
\]
We define a modified potential \(Q(z)\) by
\begin{equation}   
\label{def of example Q}
Q(z):=
\begin{cases}
Q_{\text{g}}(z), & \text{if } |z|\leq 1,\\[0.3em]
\Upsilon(z), & \text{if } 1<|z|<3,\\[0.3em]
Q_{\text{g}}(z), & \text{if } |z|\geq 3,
\end{cases}
\end{equation}
where \(\Upsilon(z)\) is a smooth, rotation-invariant function satisfying the following properties:
\begin{itemize}
\item \(\Upsilon(z)\) matches \(Q_{\text{g}}(z)\) smoothly at \(|z|=1\) and \(|z|=3\);
\item \(\check{Q}_{\text{g}}(z) \leq \Upsilon(z) \leq Q_{\text{g}}(z)\) for all \(1<|z|<3\);
\item \(\Upsilon(z) = \check{Q}_{\text{g}}(z)\) holds only for \(|z|=t_1,\dots,t_m\), for some $m\in \N_{>0}$, and $1 < t_1 < \cdots < t_m < 3$.
\end{itemize}
Then $\check{Q}=\check{Q}_{\text{g}}$ and $S^{\ast}$ is of the form \eqref{def of droplet case 1}. Hence, the function $Q$ in \eqref{def of example Q} is an example of a \hyperref[item:case1]{case~1} potential.
\end{example}

\hyperref[item:case1]{Case 1} when $m=1$ has been previously studied in \cite{ACC2023b}. We recall the one-dimensional Heine distribution, see \cite[Subsection 1.3.2]{ACC2023b}. 

\begin{definition}[One-dimensional Heine distribution]
\label{definition one dimensional heine distribution}
Let $\theta\in \R_{>0}$  and $q$ a number with $0<q<1$. A $\mathbb{N}:=\{0,1,2,\dots\}$-valued random variable $X$ is said to have a Heine distribution with parameters $(\theta,q)$, denoted $X\overset{d}{\sim}\mathrm{He}(\theta,q)$, if 
\begin{equation}
\label{def of one dim heine}
    \mathbb{P}(\{X=k\})=\frac{1}{(-\theta,q)_{\infty}}\frac{q^{\frac{1}{2}k(k-1)}\theta^k}{(q;q)_k},\qquad k\in\mathbb{N}, 
\end{equation}
where 
\begin{equation}
(z;q)_k=\prod_{i=0}^{k-1}(1-zq^i),\qquad
(z;q)_{\infty}=\prod_{i=0}^{+\infty}(1-zq^{i}). 
\end{equation}   
\end{definition}

In this setting, it is natural to ask for the limiting distribution of the number of particles lying in a small neighborhood of the circle \(\{|z|=t_1\}\).
This question was answered in \cite[Corollary~1.10]{ACC2023b}.

\begin{theorem}\textup{(\cite[Corollary 1.10]{ACC2023b})}
\label{theorem: ACC2023}
Consider the \hyperref[item:case1]{case 1} when $m=1$. 
Let $N_n$ be the number of particles lying in a small but fixed neighborhood of $\{|z|=t_1\}$. 
As $n\to+\infty$, the random variables $N_n$ converge in distribution to $\mathrm{He}(\theta\rho,\rho^2)$, where $\theta=\sqrt{\frac{\Delta Q(b_0)}{\Delta Q(t_1)}}$ and $\rho=\frac{b_0}{t_1}$.    
\end{theorem}
In particular, the number of particles near the outpost is finite in probability. Indeed, if $X\overset{d}{\sim}\mathrm{He}(\theta\rho,\rho^2)$, then we have 
\[
 \E[X]=\sum_{j=0}^{+\infty}\frac{\theta\rho^{2j+1}}{1+\theta\rho^{2j+1}},
\]
which shows that the expected number of particles near the outpost is finite.

\smallskip

By \eqref{def of decouple} and the discussion below it, it is natural to ask whether the numbers of particles lying in small but fixed neighborhoods of outposts within a spectral gap exhibit asymptotic independence.
The aim of this work is to answer this question by going beyond Theorem~\ref{theorem: ACC2023} and extend the results from \cite{ACC2023b} to the multiple-outpost \hyperref[item:case1]{case 1} and \hyperref[item:case2]{case 2}.
Our results show that the particle count near each outpost is strongly correlated with the particle counts near all other outposts, not only the nearest ones (for both \hyperref[item:case1]{case 1} and \hyperref[item:case2]{case 2}). 
Note that more complicated \hyperref[item:case2]{case 2} was not considered in \cite{ACC2023b}, even for $m=1$.

\subsection{Main results}
We start by defining the multi-dimensional Heine distribution, which generalizes Definition~\ref{definition one dimensional heine distribution}. 
\begin{definition}[Multi-dimensional Heine distribution]
\label{def of multi heine distribution}
Let $(\theta_1,\dots,\theta_m)\in \R^{m}_{>0}$ and $(q_1,\dots,q_m)\in(0,1)^m$. 
For $(\alpha_1,\dots,\alpha_m)\in\mathbb{N}^m$, a $\mathbb{N}^m$-valued random variable $\boldsymbol{X}_m=(X_1,\dots,X_m)$ is said to have a multi-dimensional Heine distribution with parameters $(\theta_1,\dots,\theta_m;q_1,\dots,q_m)$, denoted $\boldsymbol{X}_m\overset{d}{\sim}\mathrm{He}(\theta_1,\dots,\theta_m;q_1,\dots,q_m)$, if 
\begin{equation}
\label{def of multi heine distribution eq}
\mathbb{P}\bigl(\{X_1=\alpha_1,X_2=\alpha_2,\dots,X_m=\alpha_m\}\bigr)
=\frac{\theta_{1}^{\alpha_1}\cdots \theta_{m}^{\alpha_m}
\sum_{\substack{J_1,\dots,J_m\subseteq \mathbb{N},\\ k\neq \ell; |J_k|=\alpha_k,J_k\cap J_{\ell}=\emptyset}}\prod_{k=1}^m q_k^{\sum_{j\in J_k}j}}{\prod_{j=0}^{+\infty}(1+\sum_{k=1}^m \theta_kq_k^{j})}, 
\end{equation} 
where $|A|$ denotes the number of elements of a set $A$. 
\end{definition}
For any $(\alpha_1,\dots,\alpha_m)\in\mathbb{N}^m$, it is obvious that $0<\mathbb{P}\bigl(\{X_1=\alpha_1,X_2=\alpha_2,\dots,X_m=\alpha_m\}\bigr)$, and 
\begin{align*}
\sum_{\alpha_1,\dots,\alpha_m=0}^{+\infty}\mathbb{P}\bigl(\{X_1=\alpha_1,X_2=\alpha_2,\dots,X_m=\alpha_m\}\bigr)
&= 
\frac{
\sum_{\substack{J_1,\dots,J_m\subseteq \mathbb{N},\\ k\neq \ell; J_k\cap J_{\ell}=\emptyset}}\Bigl(\prod_{k=1}^m q_k^{\sum_{j\in J_k}j}\theta_k^{|J_k|}\Bigr)}{\prod_{j=0}^{+\infty}(1+\sum_{k=1}^m \theta_kq_k^{j})}
=1,
\end{align*}
which follows by extracting the coefficients $\theta_1^{\alpha_1}\theta_2^{\alpha_2}\cdots\theta_m^{\alpha_m}$ from $\prod_{j=0}^{+\infty}(1+\sum_{k=1}^m \theta_kq_k^{j})$.
Therefore, \eqref{def of multi heine distribution eq} is indeed a probability mass function. 

\begin{remark}[Consistency with the one-dimensional Heine distribution]
\label{remark one dim heine}  
The multi-dimensional Heine distribution recovers the Heine distribution when $m=1$. 
Indeed, \eqref{def of multi heine distribution eq} with $m=1$ yields  
\begin{align*}
\mathbb{P}\bigl(\{X=\alpha\}\bigr)
&=
\frac{\theta^{\alpha}
\sum_{J\subseteq \mathbb{N},|J|=\alpha}q^{\sum_{j\in J}j}}{\prod_{j=0}^{+\infty}(1+\theta q^{j})},
\end{align*}
where we set $X\equiv X_1,\theta\equiv\theta_1,J\equiv J_1$, and $\alpha=\alpha_1$. 
Any sets $J\subset \mathbb{N}$ such that $|J|=\alpha$, can be written in the form $J=\{j_1,\dots,j_{\alpha}\}$ for some non-negative integers $j_1<\cdots<j_{\alpha}$. Let us denote $i_r:=j_r-(r-1)$ for $r=1,2,\dots,\alpha$. Since $j_1+\cdots+j_{\alpha}=\sum_{r=1}^{\alpha}(i_r+(r-1))=\sum_{r=1}^{\alpha}i_r+\frac{1}{2}\alpha(\alpha-1)$, we have 
\[
\theta^{\alpha}
\sum_{J\subseteq \mathbb{N},|J|=\alpha}q^{\sum_{j\in J}j}
=
\theta^{\alpha}\sum_{0\leq j_1<\cdots<j_{\alpha}}q^{j_1+\cdots+j_{\alpha}}
=
\theta^{\alpha}q^{\frac{1}{2}\alpha(\alpha-1)}
\sum_{0\leq i_1\leq \cdots\leq i_{\alpha}}q^{i_1+\cdots+i_{\alpha}}
=\frac{\theta^{\alpha}q^{\frac{1}{2}\alpha(\alpha-1)}}{(q;q)_{\alpha}},
\]
where we have used that $\sum_{0\leq i_1\leq \cdots\leq i_{\alpha}}q^{i_1+\cdots+i_{\alpha}}=\frac{1}{(q;q)_{\alpha}}$. 
This shows that \eqref{def of multi heine distribution eq} recovers \eqref{def of one dim heine} when $m=1$. 
\end{remark}

Let us now discuss another aspect of the multi-dimensional Heine distribution. 
Let $(\theta_1,\dots,\theta_m)\in\R_{>0}^m$ and $(q_1,\dots,q_m)\in(0,1)^m$. 
For each $j\in \mathbb{N}$, we define a random variable 
\begin{equation}
    Y_j\in\{0,1,\dots,m\},
\end{equation}
whose discrete probability distribution is given by 
\begin{align}
\begin{split}
\label{def of pjk}
p_{j,0}&=\mathbb{P}\bigl(Y_j=0\bigr)=\frac{1}{1+\sum_{\ell=1}^m \theta_{\ell}q_{\ell}^j}
\\
p_{j,k}&=\mathbb{P}\bigl(Y_j=k\bigr)=\frac{\theta_kq_k^j}{1+\sum_{\ell=1}^m \theta_{\ell}q_{\ell}^j},\qquad k=1,2,\dots,m. 
\end{split}    
\end{align}
Assume moreover that the random variables $\{Y_j\}$ are independent, and define 
\begin{equation}
\label{def of Xk another}
    X_k:=\sum_{j=0}^{\infty}\mathbf{1}(Y_j=k),\qquad k=1,2,\dots,m.  
\end{equation}
Note that by the Borel-Cantelli theorem, $X_{k}<+\infty$ with probability one.

We now show that the multi-dimensional Heine distribution can be realized as the joint distribution of \((X_1,\dots,X_m)\), where \(X_k=\sum_{j\ge0}\mathbf 1(Y_j=k)\) and \(\{Y_j\}_{j\ge0}\) are independent \(\{0,1,\dots,m\}\)-valued random variables.

\begin{lemma}
The joint distribution of the random variables $X_1,\dots,X_m$ defined in \eqref{def of Xk another} is
\[
(X_1,\dots,X_m)\overset{d}{\sim}\mathrm{He}(\theta_1,\dots,\theta_m;q_1,\dots,q_m).
\]
\end{lemma}
\begin{proof}
We define  
\[
I_k:=\{j\in\mathbb{N}:Y_j=k\}. 
\]
For any $j\in \mathbb{N}$, the event $\{j\in I_k\cap I_{\ell}\}$ would imply $Y_j=k$ and $Y_j=\ell$ simultaneously, which is impossible for $k\neq \ell$. Therefore, $I_k\cap I_{\ell}=\emptyset$ with probability one. Note that 
\begin{equation}
\label{def of another Xk}
X_k=\sum_{j=0}^{\infty}\mathbf{1}(Y_j=k)=|I_k|. 
\end{equation}
Note that 
For disjoint subsets $I_1,\dots,I_m\subset\mathbb{N}$,
\[
\mathbb{P}\bigl(\{\text{$Y_j=k$ if $j\in I_k$, and $Y_j=0$ if $j\notin \cup_{k=1}^m I_k$}\}\bigr)
=
\frac{\prod_{k=1}^m \prod_{j\in I_k}\theta_k q_k^j}{\prod_{j=0}^{\infty}(1+\sum_{\ell=1}^m \theta_{\ell}q_{\ell}^j)}
=\frac{\prod_{k=1}^m \theta_k^{|I_k|}q_k^{\sum_{j\in I_k}j}}{\prod_{j=0}^{\infty}(1+\sum_{\ell=1}^m \theta_{\ell}q_{\ell}^j)}, 
\]
where we have used the fact that the family $\{Y_j\}_{j}$ is independent and \eqref{def of pjk} to make the above expression well-defined.
Note that the event $\{X_1=\alpha_1,\dots,X_m=\alpha_m\}$ for any $(\alpha_1,\dots,\alpha_m)\in \mathbb{N}^m$ is the disjoint union over all families of subsets $(I_1,\dots,I_m)$ such that $I_{k}\subset\mathbb{N}$, $|I_k|=\alpha_k$, and $I_k\cap I_{\ell}=\emptyset$ for $k\neq \ell$ for $k,\ell=1,2,\dots,m$ of the events $\{\text{$Y_j=k$ if $j\in I_k$, and $Y_j=0$ if $j\notin \cup_{k=1}^m I_k$}\}$. Therefore, 
\begin{align*}
\mathbb{P}\bigl( \{X_1=\alpha_1,\dots,X_m=\alpha_m\}\bigr)&=\sum_{\substack{I_{1},\dots,I_m\subset\mathbb{N}\\ I_{k}\cap I_{\ell}=\emptyset,k\neq\ell,|I_k|=\alpha_k}}\mathbb{P}\bigl(\{\text{$Y_j=k$ if $j\in I_k$, and $Y_j=0$ if $j\notin \cup_{k=1}^m I_k$}\}\bigr)
\\
&=
\frac{\theta_1^{\alpha_1}\cdots \theta_{m}^{\alpha_m}\sum_{\substack{I_{1},\dots,I_m\subset\mathbb{N}\\ I_{k}\cap I_{\ell}=\emptyset,k\neq\ell,|I_k|=\alpha_k}}\prod_{k=1}^m q_{k}^{\sum_{j\in I_k}j}}{\prod_{j=0}^{\infty}(1+\sum_{\ell=1}^m \theta_{\ell}q_{\ell}^j)}. 
\end{align*}
This coincides with Definition~\ref{def of multi heine distribution}. 
\end{proof}
\begin{remark}
We remark that the marginal distribution of the multi-dimensional Heine distribution is completely different from the one-dimensional Heine distribution.
Define 
\[
B_{j,k}:=\mathbf{1}(Y_j=k)\in\{0,1\}.
\]
Then, by \eqref{def of another Xk},
\[
X_k=\sum_{j=0}^{\infty}B_{j,k}
\]
is the sum of independent Bernoulli random variables with probabilities \eqref{def of pjk}.
Therefore, $X_k$ follows a Poisson-binomial distribution. 
In particular, since we assumed that the random variables $\{Y_j\}_j$ are independent, $\{B_{j,k}\}_{j}$ are also independent. Hence, we have
\[
\mathbb{E}\bigl[s^{X_k}\bigr]=\prod_{j=0}^{\infty}\E\bigl[s^{B_{j,k}}\bigr]=\prod_{j=0}^{\infty}(1-p_{j,k}+p_{j,k}s)=\prod_{j=0}^{\infty}\frac{1+\sum_{\ell\neq k}\theta_{\ell}q_{\ell}^j+\theta_kq_{k}^js}{1+\sum_{\ell=1}^m \theta_{\ell}q_{\ell}^j}. 
\]
By Definition~\ref{def of multi heine distribution} and Remark~\ref{remark one dim heine}, since the generating function of the one-dimensional Heine distribution $\mathrm{He}(\theta_k;q_k)$ with parameters $(\theta_k;q_k)$ is given by 
\[
\prod_{j=0}^{\infty}\frac{1+\theta_kq_k^js}{1+\theta_k q_k^j}, 
\]
it is only when all other parameters are set to be zero, i.e., $\theta_{\ell}=0$ for all $\ell \neq k$, that we recover the one-dimensional Heine distribution. 
This shows that the marginal distribution of the multi-dimensional Heine distribution does not degenerate to the one-dimensional Heine distribution.
As a consequence, the multi-dimensional Heine distribution should not be identified with a multinomial Poisson distribution. Moreover, each marginal follows a Poisson–binomial distribution.
\end{remark}

We now state the main results of this work, which extend those obtained in \cite[Theorem 1.2]{ACC2023b}.
Below, we use the following smooth test function.
Define
\begin{equation*}
\eta(u):=
\begin{cases}
e^{-1/u}, & u>0,\\
0, & u\le 0,
\end{cases}
\qquad
\widehat{\eta}(u):=\frac{\eta(u)}{\eta(u)+\eta(1-u)} .
\end{equation*}
Then $\widehat{\eta}\in C^\infty(\mathbb{R})$, and it satisfies $\widehat{\eta}(u)=0$ for $u\le 0$ and $\widehat{\eta}(u)=1$ for $u\ge 1$.
Let $\varepsilon>0$ be sufficiently small so that the $2\varepsilon$-neighborhoods of distinct outposts do not overlap. For each $k=1,\dots,m$, define the radial cutoff
\begin{equation*}
\chi_k(r)
:=
\widehat{\eta}\,\left(\frac{r-(t_k-\varepsilon)}{\varepsilon/2}\right)
\widehat{\eta}\,\left(\frac{(t_k+\varepsilon)-r}{\varepsilon/2}\right),
\qquad r\ge 0,
\end{equation*}
where $t_{k}$ for $k=1,2,\dots,m$ are given by \eqref{def of droplet case 1} or \eqref{def of droplet case 2}, 
and set the rotation-invariant bump function
\begin{equation}\label{eq:hk-def}
h_k(z):=\chi_k(|z|),\qquad z\in\mathbb{C}.
\end{equation}
Then $h_k\in C^\infty(\mathbb{C})$ is rotation-invariant, and
\[
h_k(z)=1 \quad \text{for } t_k-\frac{\varepsilon}{2}\le |z|\le t_k+\frac{\varepsilon}{2},
\qquad
h_k(z)=0 \quad \text{for } |z|\notin [t_k-\varepsilon,\,t_k+\varepsilon].
\]

We first consider \hyperref[item:case1]{case 1} when $S^{\ast}$ is given by \eqref{def of droplet case 1}.
We derive the asymptotic behavior of the multivariate moment generating function
of the particle numbers at each outpost.

\begin{theorem}\textup{(\hyperref[item:case1]{Case 1})}
\label{theorem: case 1}
Let $S^{\ast}$ be as in \eqref{def of droplet case 1}. 
Define
\[
\vartheta_k = \sqrt{\frac{\Delta Q(b_0)}{\Delta Q(t_k)}}
\quad \text{and} \quad
\rho_k = \frac{b_0}{t_k},
\qquad k=1,2,\dots,m.
\]
Let $h_k$ be as in \eqref{eq:hk-def} for $k=1,2,\dots,m$, and define
\[
N_{n,k} = \sum_{j=1}^n h_k(z_j), \qquad k=1,2,\dots,m.
\]
Then
\begin{equation}
\lim_{n\to+\infty}
\E\!\left[\prod_{k=1}^{m} e^{s_k N_{n,k}}\right]
=
\prod_{j=0}^{+\infty}
\frac{ 1+\sum_{k=1}^{m} e^{s_k}\vartheta_k \rho_k^{2j+1} }
     { 1+\sum_{k=1}^{m} \vartheta_k \rho_k^{2j+1} },
\end{equation}
uniformly for $|s_k|\leq \log n$, $k=1,2,\dots,m$.
In particular, by Definition~\ref{def of multi heine distribution}, as $n\to+\infty$, we have 
\begin{equation}
(N_{n,1},N_{n,2},\dots,N_{n,m})
\overset{d}{\longrightarrow}
\mathrm{He}\bigl(\vartheta_1\rho_1,\dots,\vartheta_m\rho_m;\rho_1^2,\dots,\rho_m^2\bigr).
\end{equation}
\end{theorem}

Using Theorem~\ref{theorem: case 1}, we obtain the expectation, variance, and covariance of the numbers of particles lying in small but fixed neighborhoods of an outpost. 
\begin{corollary}\textup{(\hyperref[item:case1]{Case 1}: expectation, variance, and covariance function)}
\label{corollary: case 1}
Let $p,q\in\{1,2,\dots,m\}$ be such that $p\neq q$. With the same setting as in Theorem~\ref{theorem: case 1}, as $n\to+\infty$, we have 
\begin{align}
\label{def of E case 1}
\E\bigl[N_{n,p}\bigr]&\to\sum_{j=0}^{+\infty}\frac{\vartheta_p\rho_p^{2j+1}}{1+\sum_{k=1}^m\vartheta_k\rho_k^{2j+1}},
\\
\label{def of Var case 1}
\Var \bigl[N_{n,p}\bigr]&\to\sum_{j=0}^{+\infty}\frac{\vartheta_p\rho_p^{2j+1}}{\Bigl(1+\sum_{k=1}^m\vartheta_k\rho_k^{2j+1}\Bigr)^2}\Bigl(1+\sum_{k=1;k\neq p}^m\vartheta_k\rho_k^{2j+1}\Bigr), 
\\
\label{def of cov case 1}
\Cov \bigl(N_{n,p},N_{n,q}\bigr)&\to-\sum_{j=0}^{+\infty}
\frac{(\rho_{p}\rho_{q})^{2j+1}}{\Bigl(1+\sum_{k=1}^{m}\vartheta_k\rho_{k}^{2j+1}\Bigr)^2}
\vartheta_p
\vartheta_q. 
\end{align}
\end{corollary}
This leads to the following probabilistic interpretations.
\begin{itemize}
\item Formula \eqref{def of E case 1} implies that the number of particles lying in a small neighborhood of $\{|z|=t_{k}\}$ for $k=1,2,\dots,m$ is finite in probability.
\item 
The expected number \eqref{def of E case 1} of particles at a given outpost decreases when the outpost is located farther away from the connected component of the droplet. 
More precisely, when $t_k$ increases, then $\rho_k=b_0/t_k$ decreases, which leads to a smaller limiting expected occupation number. 
\item 
By \eqref{def of cov case 1}, the particle numbers associated with different outposts are negatively correlated, reflecting an intrinsic competition between outposts outside the outermost connected component of the droplet. 
\end{itemize}

\smallskip

Finally, we provide the fluctuations for \hyperref[item:case2]{case 2}, i.e., when $S^{\ast}$ is given by \eqref{def of droplet case 2}.

\begin{theorem}\textup{(\hyperref[item:case2]{Case 2})}
\label{theorem: case 2}
Let $S^{\ast}$ be as in \eqref{def of droplet case 2}.
For
\[
x_n = M_0 n - \lfloor M_0 n \rfloor,
\qquad
M_0 = \sigma(\{ |z| \leq b_0 \}),
\]
define
\[
\widetilde{\rho}_0
\equiv\widehat{\rho}_{m+1}\equiv
\rho_0:=\frac{b_0}{a_1}\in(0,1)
\qquad
\widetilde{\rho}_k = \frac{t_k}{a_1}\in(0,1),
\qquad
\widehat{\rho}_k = \frac{b_0}{t_k}\in(0,1),\qquad k=1,2,\dots,m,
\]
and
\begin{align*}
\widetilde{\vartheta}_{0,n}&:=\sqrt{\frac{\Delta Q(a_1)}{\Delta Q(b_0)}}\rho_0^{-2x_n},\qquad \widehat{\vartheta}_{m+1,n}:=\widetilde{\vartheta}_{0,n}^{-1},
\\
\widetilde{\vartheta}_{k,n}
&:= \sqrt{\frac{\Delta Q(a_1)}{\Delta Q(t_k)}} \, \widetilde{\rho}_k^{\,-2x_n},
\qquad
\widehat{\vartheta}_{k,n}
:= \sqrt{\frac{\Delta Q(b_0)}{\Delta Q(t_k)}} \, \widehat{\rho}_k^{\,2x_n},
\qquad k=1,2,\dots,m.
\end{align*}
Let $h_0$ be a smooth, rotation-invariant test function such that $h_0\equiv1$ in neighborhoods of the boundaries $|z|=b_{0}$ of $A(a_{0},b_{0})$ and $|z|=a_{1}$ of $A(a_{1},b_{1})$ and $h_0\equiv0$ otherwise. 
Let $h_k$ be as in \eqref{eq:hk-def} for $k=1,2,\dots,m$, and set
\[
N_{n,k} = \sum_{j=1}^n h_k(z_j), \qquad k=0,1,2,\dots,m.
\]
Then as $n\to+\infty$, we have
\begin{align}
\begin{split}
\E\!\left[\prod_{k=0}^{m} e^{s_k N_{n,k}}\right]
=
\prod_{j=0}^{+\infty}
\frac{
\bigl(1+
\sum_{k=0}^m 
e^{s_k}\widetilde{\vartheta}_{k,n}
\widetilde{\rho}_k^{\,2j+1}
\bigr)
\bigl(1+
\sum_{k=1}^{m+1} 
e^{s_k}
\widehat{\vartheta}_{k,n}
\widehat{\rho}_k^{\,2j+1}
\bigr)}{
\bigl(1+
\sum_{k=0}^{m} \widetilde{\vartheta}_{k,n}
\widetilde{\rho}_k^{\,2j+1}
\bigr)
\bigl(1+
\sum_{k=1}^{m+1}\widehat{\vartheta}_{k,n}
\widehat{\rho}_k^{\,2j+1}
\bigr)}+o(1),
\end{split}    
\end{align}
uniformly for $|s_k|\leq \log n$, $k=0,1,2,\dots,m$ with $s_0\equiv s_{m+1}$.
In particular, by Definition~\ref{def of multi heine distribution}, as $n\to+\infty$,
\begin{equation}
\label{def of limiting decomposition}
(N_{n,0},N_{n,1},N_{n,2},\dots,N_{n,m})-(\boldsymbol{X}_{m+1,n}^{(1)}+\boldsymbol{X}_{m+1,n}^{(2)})
\overset{d}{\longrightarrow}
0,
\end{equation}
where $\boldsymbol{X}_{m+1,n}^{(1)}$ and $\boldsymbol{X}_{m+1,n}^{(2)}$ are independent, and
\begin{align}
\begin{split}
\label{def of Xm1 Xm2}
\boldsymbol{X}_{m+1,n}^{(1)}
&\overset{d}{\sim}
\mathrm{He}\bigl( 
\widetilde{\vartheta}_{0,n}\widetilde{\rho}_0,\dots,\widetilde{\vartheta}_{m,n}\widetilde{\rho}_m;
\widetilde{\rho}_0^{\,2},\widetilde{\rho}_1^{\,2},\dots,\widetilde{\rho}_m^{\,2}
\bigr),
\\
\boldsymbol{X}_{m+1,n}^{(2)}
&\overset{d}{\sim}
\mathrm{He}\bigl( 
\widehat{\vartheta}_{1,n}\widehat{\rho}_1,\dots,\widehat{\vartheta}_{m+1,n}\widehat{\rho}_{m+1};
\widehat{\rho}_1^{\,2},\dots,\widehat{\rho}_{m+1}^{\,2}
\bigr).
\end{split}    
\end{align}
In particular, when $m=1$, as $n\to+\infty$, we have  
\begin{equation}
    (N_{n,0},N_{n,1})-(\boldsymbol{X}_{2,n}^{(1)}+\boldsymbol{X}_{2,n}^{(2)})\overset{d}{\longrightarrow}0, 
\end{equation}
where $\boldsymbol{X}_{2,n}^{(1)}$ and $\boldsymbol{X}_{2,n}^{(2)}$ are independent, and
\begin{equation}
\label{def of Xm1 Xm2 m=1}
\boldsymbol{X}_{2,n}^{(1)}
\overset{d}{\sim}
\mathrm{He}\bigl( 
\widetilde{\vartheta}_{0,n}\widetilde{\rho}_0,\widetilde{\vartheta}_{1,n}\widetilde{\rho}_1;
\widetilde{\rho}_0^{\,2},\widetilde{\rho}_1^{\,2}
\bigr),
\qquad
\boldsymbol{X}_{2,n}^{(2)}
\overset{d}{\sim}
\mathrm{He}\bigl( 
\widehat{\vartheta}_{1,n}\widehat{\rho}_1,\widehat{\vartheta}_{0,n}\widehat{\rho}_0;
\widehat{\rho}_1^{\,2},\widehat{\rho}_2^{\,2}
\bigr).
\end{equation}
\end{theorem}

Using Theorem~\ref{theorem: case 2}, we obtain the expectation, variance, and covariance of the counting statistics near the outposts.
\begin{corollary}\textup{(\hyperref[item:case2]{Case 2}: expectation, variance, and covariance function)}
\label{corollary: case 2}
Let $p,q\in\{1,2,\dots,m\}$ be such that $p\neq q$. With the same setting as in Theorem~\ref{theorem: case 2}, as $n\to+\infty$, we have 
\begin{align}
\label{def of case E}
\E\bigl[N_{n,p}\bigr]&=\sum_{j=0}^{+\infty}
\frac{\widetilde{\vartheta}_{p,n}\widetilde{\rho}_{p}^{\,2j+1}}{1+
\sum_{k=0}^m\widetilde{\vartheta}_{k,n}\widetilde{\rho}_k^{\,2j+1}}
+\sum_{j=0}^{+\infty}
\frac{\widehat{\vartheta}_{p,n}\widehat{\rho}_p^{\,2j+1}}{1+
\sum_{k=1}^{m+1} \widehat{\vartheta}_{k,n}\widehat{\rho}_k^{\,2j+1}}+o(1),
\\
\begin{split}
\label{def of case var}
\Var \bigl[N_{n,p}\bigr]&=
\sum_{j=0}^{+\infty}
\frac{\widetilde{\vartheta}_{p,n}\widetilde{\rho}_p^{\,2j+1}}{1+
\sum_{k=0}^m \widetilde{\vartheta}_{k,n}\widetilde{\rho}_k^{\,2j+1}}
\Bigl(1+
\sum_{k=0;k\neq p}^m \widetilde{\vartheta}_{k,n}\widetilde{\rho}_k^{\,2j+1}\Bigr)
\\
&\quad
+\sum_{j=0}^{+\infty}
\frac{ \widehat{\vartheta}_{p,n}\widehat{\rho}_p^{\,2j+1}}{1+
\sum_{k=1}^{m+1}  \widehat{\vartheta}_{k,n}\widehat{\rho}_k^{\,2j+1}}
\Bigl(1+
\sum_{k=1;k\neq p}^{m+1}  \widehat{\vartheta}_{k,n}\widehat{\rho}_k^{\,2j+1}\Bigr)+o(1),
\end{split}
\\
\begin{split}
\label{def of case cov}
\Cov \bigl(N_{n,p},N_{n,q}\bigr)&=
-\sum_{j=0}^{+\infty}
\frac{\widetilde{\vartheta}_{p,n}\widetilde{\rho}_p^{\,2j+1}
\widetilde{\vartheta}_{q,n}\widetilde{\rho}_q^{\,2j+1}}{\Bigl(1+
\sum_{k=1}^{m+1} \widetilde{\vartheta}_{k,n}\widetilde{\rho}_k^{\,2j+1}\Bigr)^2}
-\sum_{j=0}^{+\infty}
\frac{ \widehat{\vartheta}_{p,n}\widehat{\rho}_p^{\,2j+1}
 \widehat{\vartheta}_{q,n}\widehat{\rho}_q^{\,2j+1}}{\Bigl(1+
\sum_{k=1}^{m+1}  \widehat{\vartheta}_{k,n}\widehat{\rho}_k^{\,2j+1}\Bigr)^2}+o(1). 
\end{split}
\end{align}
In particular, when $m=1$, as $n\to+\infty$, we have 
\begin{align*}
\E\bigl[N_{n,1}\bigr]&=\sum_{j=0}^{+\infty}
\frac{\widetilde{\vartheta}_{1,n}\widetilde{\rho}_1^{\,2j+1}}{1+
\widetilde{\vartheta}_{0,n}\widetilde{\rho}_0^{\,2j+1}+\widetilde{\vartheta}_{1,n}\widetilde{\rho}_1^{\,2j+1}}
+\sum_{j=0}^{+\infty}
\frac{\widehat{\vartheta}_{1,n}\widehat{\rho}_1^{\,2j+1}}{1+ \widehat{\vartheta}_{1,n}\widehat{\rho}_1^{\,2j+1}+\widehat{\vartheta}_{2,n}\widehat{\rho}_2^{\,2j+1}}+o(1),
\\
\Var \bigl[N_{n,1}\bigr]&=
\sum_{j=0}^{+\infty}
\frac{\widetilde{\vartheta}_{1,n}\widetilde{\rho}_1^{\,2j+1}}{1+
\widetilde{\vartheta}_{0,n}\widetilde{\rho}_0^{\,2j+1}+ \widetilde{\vartheta}_{1,n}\widetilde{\rho}_1^{\,2j+1}}
\Bigl(1+\widetilde{\vartheta}_{0,n}\widetilde{\rho}_0^{\,2j+1}\Bigr)
\\
&\quad
+\sum_{j=0}^{+\infty}
\frac{ \widehat{\vartheta}_{1,n}\widehat{\rho}_1^{\,2j+1}}{1+
\widehat{\vartheta}_{1,n}\widehat{\rho}_1^{\,2j+1}+\widehat{\vartheta}_{2,n}\widehat{\rho}_2^{\,2j+1}}
\Bigl(1+\widehat{\vartheta}_{2,n}\widehat{\rho}_2^{\,2j+1}\Bigr)+o(1),
\end{align*}
\end{corollary}
This leads to the following interpretations.
\begin{itemize}
    \item 
    In the degenerate case $m=0$ (no outposts in the gap), the result reduces to the displacement phenomenon for a single spectral gap: the fluctuation of $N_{n,0}$ is given by the sum of two independent one-dimensional Heine distributions from the two sides of the gap, see \cite[Theorem 1.7]{ACC2023b}.
    \item 
The decomposition into two independent random variables \(\boldsymbol X_{m+1,n}^{(1)}=(X_{0,n}^{(1)},X_{1,n}^{(1)},\dots,X_{m,n}^{(1)})\) and \(\boldsymbol X_{m+1,n}^{(2)}=(X_{1,n}^{(2)},\dots,X_{m,n}^{(2)},X_{m+1,n}^{(2)})\) in the limit shows that the number of particles in small neighborhoods of the outposts within a spectral gap is independently influenced by the inner and outer boundaries of the gap.
Moreover, this decomposition also captures the displacement phenomenon of particles occurring between the inner and outer boundaries across the spectral gap.
    \item In contrast with \hyperref[item:case1]{case~1}, \hyperref[item:case2]{case~2} is driven by two independent sources whose contributions add up in the limit.
The limiting covariance between the numbers of particles lying in small but fixed neighborhoods of distinct outposts is the sum of two negative terms \eqref{def of case cov} induced by the inner and outer droplets.
Since \(\boldsymbol X_{m+1,n}^{(1)}\) and \(\boldsymbol X_{m+1,n}^{(2)}\) given by \eqref{def of Xm1 Xm2} are independent, the mixed covariances vanish,
\[
\mathrm{Cov}\!\left(X_{p,n}^{(1)},X_{q,n}^{(2)}\right)=o(1),\qquad p\neq q\in\{0,1,\dots,m,m+1\},\quad n\to+\infty,
\]
and both expectations and variances decompose additively: for $p\in\{1,2,\dots,m\}$, 
\[
\mathbb E[N_{n,p}] = \mathbb E[X_{p,n}^{(1)}] + \mathbb E[X_{p,n}^{(2)}]+o(1),\qquad
\mathrm{Var}(N_{n,p})= \mathrm{Var}(X_{p,n}^{(1)}) + \mathrm{Var}(X_{p,n}^{(2)})+o(1),\qquad n\to+\infty.
\]
Note however that $\mathrm{Cov}\!(X_{p,n}^{(j)},X_{q,n}^{(j)})$ for $p\neq q\in\{0,1,\dots,m,m+1\}$ and $j=1,2$ remains of order 1 as $n\to+\infty$. 
\end{itemize}

\subsubsection*{Comments and related work}
In this work, we extend the results of Ameur, Charlier, and Cronvall to the case of an arbitrary but fixed number of outposts, and we prove that the joint distribution of the numbers of particles near the outposts converges to a multi-dimensional Heine distribution. Below we discuss related works and possible directions for future research.

The papers \cite{Byun2025,BKS2023,BKSY2025,ByunPark,C2021 FH,C2021} provide free energy expansions and precise asymptotics for the moment generating functions of counting statistics in two-dimensional Coulomb gases. It would be possible and interesting to adapt the analysis in \cite{ACC2023b}, as well as the methods developed in the present work, to this setting.

In the setting of \cite{AC2024}, Ameur and Cronvall went beyond the results of \cite{ACC2023b} and studied fluctuations of outposts for non-radially symmetric potentials. They also provided compatibility conditions for the case of a single outpost in such a general setting. It would be very interesting to identify analogous compatibility conditions in the presence of several outposts, and to analyze the corresponding fluctuations.

In the recent work \cite{AL2025}, the case where the equilibrium density $\Delta Q$ vanishes along a circle inside the droplet is investigated. This can be viewed as an analogue of \cite{BL,Cl,Eynard,Mo} in the theory of Hermitian random matrices. As mentioned in \cite{ACC2023b}, it would also be interesting to study the situation where $\Delta Q$ vanishes along a outpost.

Concerning correlation functions, Ameur and Jahic in forthcoming work \cite{AJ} provide the asymptotic analysis of the correlation kernel near an outpost. The counterpart of the Szeg\"{o} kernel arising near the edge in the absence of outposts has been studied in \cite{AC2023,ACC2023a,ACC2023c}, while the limiting kernel corresponding to the case of a single outpost is analyzed in \cite{AJ}.

Finally, it is an interesting question whether the multi-dimensional Heine distribution appears in other applications. In the one-dimensional setting, related results for the Heine distribution can be found in the work of Kemp \cite{Ke}.

\subsection*{Plan of this paper}
In Section~\ref{section: preliminary}, we recap several results in \cite{ACC2023b}. 
In Section~\ref{section: proof of theorem}, we prove Theorem~\ref{theorem: case 1} and~\ref{theorem: case 2} and Corollary~\ref{corollary: case 1} and~\ref{corollary: case 2}.  

\subsection*{Acknowledgment}
The author is grateful to Joakim Cronvall for insightful discussions during the XXI Brunel--Bielefeld Workshop, and to Yacin Ameur for useful comments on an earlier draft.
The author acknowledges support from the European Research Council (ERC), Grant Agreement No. 101115687.

\section{Preliminaries}
\label{section: preliminary}
In this section, we collect several results from \cite{ACC2023b}. 
Consider the $L^2$-space over $\C$ with norm $\|f\|^2:=\int_{\C}|f(z)|^2\,dA(z)$. 
The monic weighted orthogonal polynomial of degree $j$ in potential $Q$ is denoted 
\begin{equation}
    p_{j,sf}(z)=z^{j}e^{-\frac{n}{2}\widetilde{Q}(z)}, \qquad \widetilde{Q}(z):=Q(z)-\frac{s}{n}f(z),
\end{equation}
where $f(z):=f(|z|)$ is a rotation invariant and smooth test function. When $f\equiv0$, we simply write $p_{j}(z)\equiv p_{j,0}(z)$. Then, by Andr\'{r}eief’s identity (see e.g., \cite{Fo}), we have 
\begin{equation}
\label{def of andreif}
\E\Bigl[e^{s\sum_{j=1}^n f(z_j)}\Bigr]=\prod_{j=0}^{n-1}\frac{\|p_{j,sf}\|^2}{\|p_{j}\|^2}.
\end{equation}     
A crucial method to compute the above right hand side is the Laplace method developed in \cite{ACC2023b}. 

We recall the facts of the local peak sets. 
For a fixed number $\tau$ with $0\leq \tau\leq 1$, put 
\[
    g_{\tau}(r):=q(r)-2\tau\log r. 
\]
If $r=r_{\tau}$ is a solution to $g_{\tau}'(r)=0$, and if $Q$ is smooth ar $r$, then 
\[
    rq'(r)=2\tau,\qquad q_{\tau}''(r)=4\Delta Q(r). 
\]
The solutions $r=r_{\tau}$, which give local minima for $g_{\tau}$, are called {\it local peak points}. 
We take $\mathcal{N}$ small enough so that $Q$ is $C^6$-smooth and strictly subharmonic on the set $\{z=re^{i\theta}:r\in\overline{\mathcal{N}},0\leq \theta\leq 2\pi\}$. 

We denote the totality of local peak points in $\mathcal{N}$ by
\[
    \mathrm{LP}(\tau):=\{r\in\mathcal{N}:g_{\tau}'(r)=0\},
\]
and we set $P(\tau)$ of global peak points by
\[
    \mathrm{P}(\tau):=\{r\geq0:g_{\tau}(r)=B_{\tau}\}.
\]
Given the assumption on $g_{\tau}''(r_{\tau})=4\Delta Q(r_{\tau})>0$, all points in $\mathrm{LP}(\tau)$ are strict local minima.
As a consequence, there is at most one local peak point $r_{\tau}$ in the vicinity of a given connected component $C$ of $S^{\ast}\cap[0,+\infty)$. 
We set 
\begin{equation}
    \delta_n:=C\frac{\log n}{n},\qquad \epsilon_n:=\sqrt{\delta_n},
\end{equation}
where $C$ is a large constant, and define the set of {\it significant local peak points} to be 
\begin{equation}
\mathrm{SLP}(\tau):=\{r\in\mathrm{LP}(\tau):g_{\tau}(r)<B_{\tau}+\delta_n\}.     
\end{equation}
In the following, we call a number $\tau\in[0,1]$ is a {\it branching value} if the peak set $\mathrm{P}(\tau)$ consists of at least two points. 
The values $M_0,M_1,\dots,M_{N-1}$ are branching values, and these are all in the open interval $(0,1)$. 
The value $\tau_0$ is a branching value if there is a outpost $|z|=c$ with $c<a_0$ and $\tau=1$ is a branching value if there is an outpost with $c>b_N$. 

With these sets, we recall the following two facts.

\begin{lemma}[Lemma 2.8 in \cite{ACC2023b}]
\label{lemma: several outposts sig solutions}
$\mathrm{SLP}(\tau)$ consists of a single point $r=r_{\tau}$ when $\tau$ is sufficiently for away from the branching values, in the sense that $|\tau-M_{\nu}|\geq c>0$ for all $\nu$. If $\tau$ is close to $M_{\nu}$, $|\tau-M_{\nu}|<c$, there might be several significant local peaks (the end-points $b_{\nu},a_{\nu+1}$ and possibly some shallow points in between if $0\leq \nu\leq N-1$).   
\end{lemma}

\begin{lemma}[Lemma 2.9 in \cite{ACC2023b}]
If $|\tau-M_{\nu}|\geq C\frac{\log n}{n}$ for all branching values $M_{\nu}$, where $C$ is large enough, then $\mathrm{SLP}(\tau)$ consists of a single point in the interior of $S$.     
\end{lemma}

As summary, the global peak set $\mathrm P(\tau)$ describes the radial locations where outposts may occur from the viewpoint of the obstacle problem. 
It consists of the relevant boundary points of the droplet and, at branching values, possibly additional shallow points in the spectral gaps. The local peak set $\mathrm{LP}(\tau)=\{r\in\mathcal N:g_\tau'(r)=0\}$ consists of the local minima of $g_\tau$. These points govern the Laplace-type asymptotics of the integrals and provide the candidates for exponentially dominant contributions. The set of significant local peaks $\mathrm{SLP}(\tau)\subset\mathrm{LP}(\tau)$ collects those local peaks whose values of $g_\tau$ are within order $n^{-1}\log n$ of the global minimum. 
These are precisely the points that contribute at leading order to the asymptotics for finite $n$.

We write 
\begin{equation}
\label{def of Itau n}
    I_{\tau}(n):=2\int_0^{+\infty}r^{1+2\alpha}e^{sh(r)}e^{-ng_{\tau}(r)}\,dr. 
\end{equation}

\begin{lemma}[Lemma 2.10 in \cite{ACC2023b}]
\label{lemma: siginificant solution lemma}
For each $0\leq\tau\leq 1$, define 
\begin{equation}
    J_{\tau}\equiv J_{n,\tau}:=\{r\geq 0:\mathrm{dist}(r,\mathrm{SLP}(\tau))<\epsilon_n\}.
\end{equation}
Also, write 
\begin{equation}
I_{\tau}^{\#}(n):=2\int_{J_{\tau}}r^{1+2\alpha}e^{sh(r)}e^{-ng_{\tau}(r)}\,dr.    
\end{equation}
Then if $C$ is large enough, the integral \eqref{def of Itau n} satisfies 
\begin{equation}
    I_{\tau}(n)=I_{\tau}^{\#}(n)\cdot(1+\mathcal{O}(n^{-100})),
\end{equation}
where the error term is uniform for $0\leq\tau\leq 1$ and all real $s$ with $|s|\leq \log n$.
\end{lemma}

Therefore, to show Theorem~\ref{theorem: case 1} and~\ref{theorem: case 2}, Lemma~\ref{lemma: siginificant solution lemma} allows us to focus on fluctuations for multiple outposts when we compute \eqref{def of andreif} as $n\to+\infty$.

\section{Proofs of Theorem~\ref{theorem: case 1} and~\ref{theorem: case 2}}
\label{section: proof of theorem}
\subsection{Proof of Theorem~\ref{theorem: case 1}}
We imitate \cite[Proof of Theorem 1.8]{ACC2023b}. 
We consider $S^{\ast}$ as \eqref{def of droplet case 1}. 

\begin{lemma}\label{lemma: annulus outside several outposts}
Let $h$ be a smooth test function to be 1 on neighborhood of each outpost $\{|z_k|=t_k\}$ for $k=1,2,\dots,m$ and to be 0 otherwise.
For $k=1,2,\dots,m$, write 
\begin{equation}
\rho_k:=\frac{b}{t_k},\qquad
\theta_k:=\sqrt{\frac{\Delta Q(b)}{\Delta Q(t_k)}},\qquad
c_k:=h(t_k)-h(b),\qquad
\mu(s):=e^{c_ks}\theta_k. 
\end{equation}
Then, as $n\to+\infty$, we have \begin{equation}
\lim_{n\to+\infty}\E\Bigl[e^{s\sum_{j=1}^nh(z_j)}\Bigr]
=\prod_{j=0}^{+\infty}
\frac{1+\sum_{k=1}^{m}\mu_k(s)\rho_k^{2j+1}}{ 1+\sum_{k=1}^{m}\mu_k(0)\rho_k^{2j+1}}. 
\end{equation}
uniformly for $|s|\leq \log n$. 
\end{lemma}

\begin{proof}
Let $L_n:=C\log n$, where $C$ is large enough. 
We write   
 \begin{equation}
     h_{k,j}:=\int_{\{|r-r_{k,j}|<\epsilon_n\}}2r e^{sh(r)}e^{-ng_{\tau(j)}(r)}dr,
 \end{equation}
  where $k\in\{0,1,\dots,m\}$ if $j\geq n-L_n$ while $k=0$ if $j<n-L_n$. By Lemma~\ref{lemma: siginificant solution lemma}, as $n\to+\infty$, we have 
\begin{align*}
    h_{j}=\begin{cases}
   h_{0,j}\cdot(1+\mathcal{O}(n^{-100})),     & \mbox{if } j<n-L_n, \\
       (h_{0,j}+h_{1,j}+\cdots+h_{m,j})\cdot(1+\mathcal{O}(n^{-100})),     & \mbox{if } j\geq n-L_n,
    \end{cases}
\end{align*}
which give rise to 
\begin{align*}
\sum_{j=0}^{n-1}\log h_j
=\sum_{j=0}^{n-1}\log h_{0,j}
+\sum_{j=n-L_n}^{n-1}\log\Bigl(1+\frac{h_{1,j}+\cdots+h_{m,j}}{h_{0,j}}\Bigr)+\mathcal{O}\Bigl( \frac{1}{n^{99}}\Bigr).
\end{align*}
Let
\[
\log \E\Bigl[e^{s\sum_{j=1}^nh(z_j)}\Bigr]
=\sum_{j=n-L_n}^{n-1}\log\Bigl(1+\frac{h_{1,j}+\cdots+h_{m,j}}{h_{0,j}}\Bigr).
\]
Then, we have 
\begin{align*}
\log \E\Bigl[e^{s\sum_{j=1}^nh(z_j)}\Bigr]
&=
\sum_{j=n-L_n}^{n-1}\log\Bigl(
1+\sum_{k=1}^{m}\rho_k^{2(n-j)-1}\mu_k(s)
\Bigr)
+\mathcal{O}\Bigl(\frac{1+|s|}{n}\Bigr),
\\
&=
\sum_{j=0}^{L_n-1}\log\Bigl(
1+\sum_{k=1}^{m}\rho_k^{2j+1}\mu_k(s)
\Bigr)
+\mathcal{O}\Bigl(\frac{1+|s|}{n}\Bigr)
\\
&\to
\sum_{j=0}^{+\infty}\log\Bigl(
1+\sum_{k=1}^{m}\rho_k^{2j+1}\mu_k(s)
\Bigr),\qquad n\to+\infty,
\end{align*}
where $\rho_k=\frac{b}{t_k}$ and $\mu_k(s)=e^{s(h(t_k)-h(b))}\sqrt{\frac{\Delta Q(b)}{\Delta Q(t_k)}}$. 
\end{proof}
For $k=1,2,\dots,m$, let us pick a smooth test function $h_{k}$ as \eqref{eq:hk-def}.
Let us define 
\[
N_{n,k}:=\sum_{j=1}^n h_{k}(z_j),\qquad k=1,2,\dots,m, 
\]
which corresponds to the random variable of the number of particles which are found in a vicinity of the one outpost.  
For $\Vec{\boldsymbol{s}}=(s_1,\dots,s_m)\in\R^m$, define the multivariate moment generating function 
\begin{equation}
    G_{n,m}(\Vec{\boldsymbol{s}}):=\E\Big[\prod_{k=1}^me^{s_kN_{n,k}}\Bigr], 
\end{equation}
Then, by Lemma~\ref{lemma: annulus outside several outposts}, as $n\to+\infty$, we have 
\[
G_m(\Vec{\boldsymbol{s}}):=
\lim_{n\to+\infty}
G_{n,m}(\Vec{\boldsymbol{s}})=\prod_{j=0}^{+\infty}\frac{ 1+\sum_{k=1}^{m}\mu_k(s_k)\rho_k^{2j+1}}{ 1+\sum_{k=1}^{m}\mu_k(0)\rho_k^{2j+1}}. 
\]
Here, the convergence is uniform for $\Vec{\boldsymbol{s}}$ in compact subsets of $\R^m$. 
In particular, the random variable $\boldsymbol{N}_{n}:=(N_{n,1},N_{n,2},\dots,N_{n,m})$ converges in distribution to $\mathrm{He}(\vartheta_1\rho_1,\vartheta_2\rho_2,\dots,\vartheta_m\rho_m;\rho_1^2,\rho_2^2,\dots,\rho_m^2)=(N_1,N_2,\dots,N_m)$, where $\vartheta_k:=\sqrt{\frac{\Delta Q(b)}{\Delta Q(t_k)}}$ for $k=1,2,\dots,m$.
Note that 
\[
\partial_{s_{p}}\log G_m(\Vec{\boldsymbol{s}})\Bigr|_{\Vec{\boldsymbol{s}}=\Vec{\boldsymbol{0}}}
=
\sum_{j=0}^{+\infty}\frac{\vartheta_p\rho_p^{2j+1}}{1+\sum_{k=1}^m\vartheta_k\rho_k^{2j+1}}, 
\]
which leads to the expectation $\E[N_p]$ for $p\in\{1,2,\dots,m\}$.
Note also that 
\[
\partial_{s_{p}}^2\log G_m(\Vec{\boldsymbol{s}})\Bigr|_{\Vec{\boldsymbol{s}}=\Vec{\boldsymbol{0}}}
=
\sum_{j=0}^{+\infty}\frac{\vartheta_p\rho_p^{2j+1}}{\bigl(1+\sum_{k=1}^m\vartheta_k\rho_k^{2j+1}\bigr)^2}\Bigl(1+\sum_{k=1;k\neq p}^m\vartheta_k\rho_k^{2j+1}\Bigr), 
\]
which leads to the variance $\Var[N_p]$ for $p\in\{1,2,\dots,m\}$.
Finally, the covariance between $(N_{p},N_{q})$ for $p\neq q$, which corresponds to the limiting covariance function of the distinct outposts, can be computed as  
\[
\partial_{s_{p}}\partial_{s_{q}}\log G_m(\Vec{\boldsymbol{s}})\Bigr|_{\vec{\boldsymbol{s}}=\Vec{\boldsymbol{0}}}
=-\sum_{j=0}^{+\infty}
\frac{\vartheta_p\vartheta_q(\rho_{p}\rho_{q})^{2j+1}}{\bigl(1+\sum_{k=1}^{m}\vartheta_k\rho_{k}^{2j+1}\bigr)^2}.
\]

\subsection{Proof of Theorem~~\ref{theorem: case 2}}
We consider $S^{\ast}$ as \eqref{def of droplet case 1}. 
Next we consider the case of several outposts inside the droplet. 
By \cite[Subsection 3.3 for $N=1$]{ACC2023b},  
\begin{equation}
\label{def of partition annulus inside}
\sum_{j=0}^{n-1}\log h_j=\sum_{j=0}^{m_0-1}\log h_{0,j}+\sum_{j=m_0}^{n-1}\log h_{n,j}+T_n,  
\end{equation}
where 
\[
T_n:=\sum_{j=m_0}^{m_0+L_n}\log\Bigl(1+\frac{h_{0,j}+h_{1,j}+\cdots+h_{m,j}}{h_{n,j}}\Bigr) 
+
\sum_{j=m_0-L_n}^{m_0-1}\log\Bigl(1+\frac{h_{1,j}+h_{2,j}+\cdots+h_{m,j}+h_{n,j}}{h_{0,j}}\Bigr).
\]
As in the case of several outposts outside the droplet, by localizing the potential $Q$ to $\widetilde{Q}$ in a proper way, it suffices to compute the large-$n$ asymptotics of $T_n$. 
\begin{lemma}
\label{lemma: localized annulus case}
Let 
\begin{align*}
\mu_0(s)&:=\sqrt{\frac{\Delta Q(b_0)}{\Delta Q(a_1)}}e^{s(h(a_1)-h(b_0))}\Bigl(\frac{b_0}{a_1}\Bigr)^{2x_0},\qquad
\rho_0:=\frac{b_0}{a_1}, \\
\eta_k(s)&:=\sqrt{\frac{\Delta Q(a_1)}{\Delta Q(t_k)}}
e^{s(h(t_k)-h(a_1))}
\Bigl(\frac{a_1}{t_k}\Bigr)^{2x_0},\qquad
\xi_k:=\frac{t_k}{a_1}
\\
\widetilde{\eta}_k(s)&:=\sqrt{\frac{\Delta Q(b_0)}{\Delta Q(t_k)}}
e^{s(h(t_k)-h(b_0))}
\Bigl(\frac{b_0}{t_k}\Bigr)^{2x_0},
\qquad
\widetilde{\xi}_k:=\frac{b_0}{t_k}.
\end{align*}
Then, for $|s|\leq \log n$, as $n\to+\infty$, we have 
\begin{align*}
T_n&=   
\bigg[
\sum_{j=0}^{+\infty}
\log\Bigl(1+\sum_{k=1}^m\eta_k(s)\xi_k^{2j+1}+\mu_0(s)^{-1}\rho_0^{2j+1}\Bigr)
\\
&\qquad
+
\sum_{j=0}^{+\infty}
\log\Bigl(1+\sum_{k=1}^m\widetilde{\eta}_k(s)\widetilde{\xi}
_k^{2j+1}+\mu_0(s)\rho_0^{2j+1}\Bigr)
\bigg]\Bigl(1+\mathcal{O}\Bigl(\frac{1+|s|}{n}\Bigr)\Bigr). 
\end{align*}
\end{lemma}

\begin{proof}
The proof is done in a similar way of \cite[Lemma 3.10]{ACC2023b}. 
We omit the details. 

\end{proof}
Let $h_k$ be \eqref{eq:hk-def}.
We define 
\[
    N_{n,k}:=\sum_{j=1}^n h_k(z_j), \qquad k=0,1,\dots,m+1, 
\]
and for $\Vec{\boldsymbol{s}}:=(s_0,s_1,\dots,s_m,s_{m+1})\in \R^{m+2}$, let us denote 
\[
G_{n,m}^{(\mathrm{an})}(\Vec{\boldsymbol{s}}):=\E\Bigl[\prod_{k=0}^{m+1}e^{s_k N_{n,k}}\Bigr].
\]
By Lemma~\ref{lemma: localized annulus case}, we have 
\begin{align*}
G_{m}^{(\mathrm{an})}(\Vec{\boldsymbol{s}}) 
&
=
\prod_{j=0}^{+\infty}
\frac{
\big(1+
\sum_{k=0}^m 
e^{s_k-s_1}
\widetilde{\vartheta}_{k}
\widetilde{\rho}_{k}^{\,2j+1}
\big)\big(1+
\sum_{k=1}^{m+1}
e^{s_k-s_0}
\widehat{\vartheta}_{k}
\widehat{\rho}_{k}^{\,2j+1}
\big)}{
\big(1+
\sum_{k=0}^m 
\widetilde{\vartheta}_{k}
\widetilde{\rho}_{k}^{\,2j+1}
\big)\big(1+
\sum_{k=1}^{m+1}
\widehat{\vartheta}_{k}
\widehat{\rho}_{k}^{\,2j+1}
\big)}+o(1),   
\end{align*}
uniformly for $\Vec{\boldsymbol{s}}$ in compact subsets of $\R^{m+2}$. 
For $p=1,2,\dots,m$, we have 
\begin{align*}
\partial_{s_p}\log G_{n,m}^{(\mathrm{an})}(\Vec{\boldsymbol{s}})\Bigr|_{\Vec{\boldsymbol{s}}=\Vec{\boldsymbol{0}}}&=
\sum_{j=0}^{+\infty}
\frac{\widetilde{\vartheta}_{p}
\widetilde{\rho}_{p}^{\,2j+1}}{1+
\sum_{k=0}^m \widetilde{\vartheta}_{k}
\widetilde{\rho}_{k}^{\,2j+1}}
+\sum_{j=0}^{+\infty}
\frac{\widehat{\vartheta}_{p}
\widehat{\rho}_{p}^{\,2j+1}}{1+
\sum_{k=1}^{m+1} \widehat{\vartheta}_{k}
\widehat{\rho}_{k}^{\,2j+1}}+o(1). 
\end{align*}
For $p=1,2,\dots,m$, we have 
\begin{align*}
 \partial_{s_p}^2\log G_{n,m}^{(\mathrm{an})}(\Vec{\boldsymbol{s}})\Bigr|_{\Vec{\boldsymbol{s}}=\Vec{\boldsymbol{0}}}&=
\sum_{j=0}^{+\infty}
\frac{\widetilde{\vartheta}_{p}
\widetilde{\rho}_{p}^{\,2j+1}}{1+
\sum_{k=0}^m \widetilde{\vartheta}_{k}
\widetilde{\rho}_{k}^{\,2j+1}}
\Bigl(1+
\sum_{k=0;k\neq p}^m \widetilde{\vartheta}_{k}
\widetilde{\rho}_{k}^{\,2j+1}\Bigr)
\\
&\qquad
+\sum_{j=0}^{+\infty}
\frac{\widehat{\vartheta}_{p}
\widehat{\rho}_{p}^{\,2j+1}}{1+
\sum_{k=1}^{m+1} \widehat{\vartheta}_{k}
\widehat{\rho}_{k}^{\,2j+1}}
\Bigl(1+
\sum_{k=1;k\neq p}^{m+1} \widehat{\vartheta}_{k}
\widehat{\rho}_{k}^{\,2j+1}\Bigr)+o(1). 
\end{align*}   
For $p\neq q\in\{1,2,\dots,m\}$, we have 
\begin{align*}
\partial_{s_p}\partial_{s_q}\log G_{n,m}^{(\mathrm{an})}(\Vec{\boldsymbol{s}})\Bigr|_{\Vec{\boldsymbol{s}}=\Vec{\boldsymbol{0}}}&=
-\sum_{j=0}^{+\infty}
\frac{\widetilde{\vartheta}_{p}
\widetilde{\rho}_{p}^{\,2j+1}
\widetilde{\vartheta}_{q}
\widetilde{\rho}_{q}^{\,2j+1}
}{\bigl(1+
\sum_{k=1}^{m+1}\widetilde{\vartheta}_{k}
\widetilde{\rho}_{k}^{\,2j+1}\bigr)^2}
-\sum_{j=0}^{+\infty}
\frac{\widehat{\vartheta}_{p}
\widehat{\rho}_{p}^{\,2j+1}
\widehat{\vartheta}_{q}
\widehat{\rho}_{q}^{\,2j+1}
}{\bigl(1+
\sum_{k=1}^{m+1} \widehat{\vartheta}_{k}
\widehat{\rho}_{k}^{\,2j+1}\bigr)^2}+o(1),
\end{align*}
which provides the covariance between the outposts.

\end{document}